\documentclass[a4paper,10pt,notitlepage,reqno]{article}
\usepackage[english]{babel}
\usepackage[latin1]{inputenc}
\usepackage[T1]{fontenc}
\usepackage{amssymb,amsthm,amsmath} 
\usepackage{mathtools}
\usepackage{mathrsfs}
\usepackage{enumerate}
\usepackage{authblk}
\usepackage{color}
\usepackage{geometry}
\usepackage{hyperref}
\theoremstyle{plain} 
\newtheorem{theorem}{Theorem}[section]
\newtheorem{lemma}{Lemma}[section]
\newtheorem{proposition}{Proposition}[section]
\newtheorem{corollary}{Corollary}[section]

\theoremstyle{definition}
\newtheorem{definition}{Definition}[section]

\theoremstyle{remark}
\newtheorem{remark}{Remark}[section]

\DeclarePairedDelimiter{\abs}{\lvert}{\rvert} 
\DeclarePairedDelimiter{\norm}{\lVert}{\rVert}  

\DeclareMathOperator{\divergenza}{div}
\renewcommand{\div}{\divergenza}

\DeclareMathOperator{\sgn}{sgn}

\newcommand{\R}{\mathbb{R}}
\newcommand{\C}{\mathbb{C}}

\newcommand{\normeq}[1]{{\left\vert\kern-0.25ex\left\vert\kern-0.25ex\left\vert #1 
    \right\vert\kern-0.25ex\right\vert\kern-0.25ex\right\vert}}

\newenvironment{system}%
{\left\lbrace\begin{array}{r@{\hspace{1mm}}ll}}%
{\end{array}\right.}

\title{\textbf{Bounds on eigenvalues of perturbed Lamé operators with complex potentials}}

\author{Lucrezia Cossetti}

\affil{Karlsruher Institut f\"ur Technologie, Englerstra{\ss}e 2, 76131 Karlsruhe, Germany; lucrezia.cossetti@kit.edu.}

\begin{document}

\date{}

\maketitle
\vspace{-1cm}


\begin{abstract}
	\noindent
	Several recent papers have focused their attention in proving the correct analogue to the Lieb-Thirring inequalities for non self-adjoint operators and in finding bounds on the distribution of their eigenvalues in the complex plane. This paper provides some improvement in the state of the art in this topic. Precisely, we address the question of finding quantitative bounds on the discrete spectrum of the perturbed Lamé operator of elasticity $-\Delta^\ast + V$ in terms of $L^p$-norms of the potential. Original results within the self-adjoint framework are provided too.
\end{abstract}


\section{Introduction}
This paper is devoted to providing bounds on the location of discrete eigenvalues of operators of the form
\begin{equation*}
	-\Delta^\ast + V
\end{equation*}  
acting on $[L^2(\R^d)]^d,$ the Hilbert space of vector fields with components in $L^2(\R^d)$.

Here $-\Delta^\ast$ denotes the Lamé operator of elasticity, that is a linear, symmetric, differential operator of second order that acts on smooth $L^2$-vector fields $u=(u_1,u_2,\dots,u_d)$ on $\R^d,$ say $[C^\infty_0(\R^d)]^d,$ in this way:
\begin{equation}\label{Lamé_operator}
	-\Delta^\ast u:= -\mu \Delta u 
									 - (\lambda + \mu) \nabla \div u \qquad \text{with}\qquad \Delta u:=(\Delta u_1, \Delta u_2, \dots, \Delta u_d). 
\end{equation}
The two material-dependent constants $\lambda,\mu\in \R$ (tipically called Lamé's coefficients) are assumed to satisfy the standard conditions
\begin{equation}\label{Lamé_parameter}
		\mu>0, \, 
		\lambda +2\mu>0, 
\end{equation}
which guarantee the strong ellipticity of $- \Delta^\ast$ (refer to~\cite{Lucrezia} for a brief summary in matter of ellipticity for systems of second order differential operators). 

$V$ represents the operator of multiplication by a (possibly) \emph{complex-valued} generating function $V\colon \R^d \to \C,$ this leads us to working within a \emph{non self-adjoint} setting.

\medskip
The distribution of eigenvalues of \emph{self-adjoint} operators has been intensively studied for several decades and nowadays the usage of powerful techniques such as spectral theorem and variational principles have become the standard approach for addressing this issue. However, as these tools are no longer available in a non self-adjoint setting, the generalization of spectral bounds to the non self-adjoint framework is not straightforward and requires a diverse strategy.

Nevertheless, a systematic, albeit recent, approach to the study of eigenvalue estimates for perturbed operators with complex-valued potentials, in particular for Schr\"odinger operators $-\Delta + V$, has been successfully developed and has already a bibliography. It is very well known consequence of Sobolev inequalities that if $V$ is \emph{real-valued} then the distance from the origin of every eigenvalue $z$ lying in the negative semi-axis (discrete eigenvalue) can be bounded in terms of $L^p$-norm of the potential (see~\cite{Keller,L_T} and~\cite{C_F_L} for a more recent improvement), more precisely the following bound 
\begin{equation}\label{real-valued-V-bound}
	\abs{z}^\gamma\leq C_{\gamma, d} \norm{V}_{L^{\gamma + \frac{d}{2}}(\R^d)}^{\gamma+ \frac{d}{2}}
\end{equation}
holds true for every $\gamma \geq 1/2$ if $d=1,$ $\gamma>0$ if $d=2$ and $\gamma\geq 0$ if $d\geq 3,$ with $C_{\gamma, d}$ independent of $V.$  

In 2001, Abramov, Aslanyan and Davies~\cite{A_A_D} generalized the previous result to \emph{complex-valued} potentials showing that every discrete eigenvalue $z$ of the \emph{one-dimensional} Schr\"odinger operator $-d^2/dx^2 + V,$ i.e. $z\in \C \setminus [0,\infty),$ lies in the complex plane within a $1/4 \norm{V}_{L^1(\R)}^2$ radius of the origin.  In order to overcome the lack of the aforementioned self-adjoint-based tools, the authors introduced a strategy built on the Birman-Schwinger principle (see, for instance, Birman's paper~\cite{Birman2}), which has permeated all the subsequent works on this topic.
Ten years later Frank~\cite{Frank} went beyond the one dimensional restriction proving the validity of~\eqref{real-valued-V-bound} for any $d\geq 2$ and $0<\gamma\leq 1/2,$ covering also the endpoint case $\gamma=0$ if $d\geq 3.$  His notable accomplishment derives from replacing a pointwise bound for the Green function of $-d^2/dx^2 -z,$ used in~\cite{A_A_D} and highly sensitive of the one dimensional framework, with the much deeper uniform resolvent estimates due to Kenig, Ruiz and Sogge~\cite{K_R_S} (see below for further details).    
The result in~\cite{Frank} partially proved the conjecture in~\cite{La_Sa} according to which~\eqref{real-valued-V-bound} holds for any $d\geq 2$ and $0<\gamma\leq d/2$, leaving opened only the case $1/2<\gamma<d/2.$ This range was covered later by Frank and Simon in~\cite{F_S} for \emph{radial potential}. In the same work the general case was investigated too. In these respects the authors provided the construction of a sequence of real-valued potentials $V_n$ with $\norm{V_n}_p \to 0$ for $d\geq 2$ and for any $p>(d+1)/2$ such that $-\Delta + V_n$ has eigenvalue 1. Even if this does not disprove the conjecture, as it is stated for discrete eigenvalues only, the result is still relevant in view of the recent generalization in~\cite{F_S} of Frank's work~\cite{Frank} to positive (embedded) eigenvalues. 
Recent related results can be also found in~\cite{F}. 

\medskip
In passing, let us stress that the extensive bibliography devoted to spectral bounds in a non self-adjoint context that the appearance of the cited work~\cite{A_A_D} stimulated, did not remain long confined to the class of Schr\"odinger operators. Indeed subsequent investigations have shown the robustness of the approach introduced in~\cite{A_A_D} by fruitfully testing it to several other models. Without attempting to be exhaustive we mention~\cite{Cuenin} in which the issue of eigenvalue bounds was studied in connection with fractional Schr\"odinger operators and Dirac operators with complex-valued potentials (see also~\cite{F_L_S} and~\cite{Cuenin2}). In the other way round, namely within the higher-order partial differential operators setting of problems, goes the recent work~\cite{I_K_L} in which the authors extended bounds in~\cite{A_A_D} to biharmonic operators.

\medskip
The main aim of our paper is to investigate on spectral properties in the elasticity setting, specifically providing bounds on the distribution of discrete (possibly complex) eigenvalues of Lamé operators~\eqref{Lamé_operator} in terms of $L^p$-norm of the potential. In other words we want to show the validity of a suitable analogous of~\eqref{real-valued-V-bound} in this diverse framework.      

A first motivation to our purpose is the trivial observation that in $d=1$ our operator $-\Delta^\ast$ turns into a scalar differential operator and, even more relevant, it is simply a multiple of the Laplacian, more precisely
\begin{equation*}
	-\Delta^\ast:= -\mu \frac{d^2}{dx^2} -(\lambda + \mu) \frac{d^2}{dx^2}= -(\lambda + 2\mu)\frac{d^2}{dx^2}.
\end{equation*}

Therefore, by virtue of the aforementioned analogous theorem in~\cite{A_A_D} (\emph{cfr} Theorem 4) for non self-adjoint Schr\"odinger operators $-d^2/dx^2 + V,$ the following result comes as no surprise and it is indeed its straightforward consequence.
\begin{theorem}\label{thm:Lamé_1d}
	Let $d=1.$ Then any eigenvalue $z\in \C \setminus [0,\infty)$ of the perturbed Lamé operator $-\Delta^\ast + V$ satisfies
	\begin{equation}\label{bound_1d}
		\abs{z}^{1/2}\leq \frac{1}{2 \sqrt{\lambda + 2\mu}} \norm{V}_{L^1(\R)}. 
	\end{equation}
\end{theorem}
The validity of Theorem~\ref{thm:Lamé_1d}, compared with Theorem 4 in~\cite{A_A_D}, together with the aforementioned extensions to higher dimensional Schr\"odinger operators, motivate the question of whether an estimate similar to~\eqref{bound_1d} holds true in $d\geq 2.$ 

As already said, starting from the celebrated paper of Abramov, Aslanyan and Davies~\cite{A_A_D}, the usage of the Birman-Schwinger principle  has been recognized as a crucial tool to get this type of bounds in a non self-adjoint setting of problem. Roughly speaking this permits to re-phrase conveniently the eigenvalues problem associated to a perturbed operator, say $H_0 + V,$ in terms of the eigenvalues problem of an integral operator suitably related to the latter. Precisely, the following proposition holds true.  
\begin{proposition}[Birman-Schwinger principle]\label{Birman-Schwinger}
	Let  $z\notin \sigma(H_0).$ Then 
	\begin{equation*}
		z \in \sigma_p(H_0 + V) \quad \Longleftrightarrow \quad -1 \in \sigma_p(V_{1/2} (H_0-z)^{-1} \abs{V}^{1/2}),
	\end{equation*}
	 with $V_{1/2}:= \abs{V}^{1/2} \sgn(V).$ (If $z\in \C,$ then $\sgn(z)$ represents the complex signum function defined by $\sgn(z):=z/\abs{z}$ if $z\neq 0$ and $\sgn(0):=0$).
\end{proposition} 

In the context of elasticity, $H_0$ is replaced by the Lamé operator $-\Delta^\ast,$ therefore the corresponding Birman-Schwinger operator has the form $V_{1/2} (-\Delta^\ast-z)^{-1} \abs{V}^{1/2}$ and makes sense if $z \in \C \setminus [0, \infty)$ as $\sigma(-\Delta^\ast)=[0,\infty).$
This leads to the need for an explicit expression of the resolvent operator associated with $-\Delta^\ast.$ 
In this regards it turns out that $(-\Delta^\ast -z)^{-1}$ has a favorable form, \emph{i.e.},
\begin{equation}\label{favorable_form}
	(-\Delta^\ast -z)^{-1} g= \frac{1}{\mu} \big(-\Delta -\tfrac{z}{\mu}\big)^{-1} \mathcal{P} g + \frac{1}{\lambda + 2\mu} \big(-\Delta - \tfrac{z}{\lambda + 2\mu}\big)^{-1} (I- \mathcal{P}) g,
\end{equation}
where $\mathcal{P}$ is the so-called Leray projection operator, customarily used in elasticity to decompose $[L^2(\R^d)]^d$ vector fields into a divergence-free component plus a gradient (refer to the preliminary section for further details).

This expression shows that, as soon as the decomposition, also known as Helmholtz decomposition, $g= \mathcal{P}g + (I-\mathcal{P})g$  is operated, the resolvent operator $(-\Delta^\ast - z)^{-1}$ splits into a sum of two vector-valued resolvent operators associated with the Laplacian acting, respectively, on the components $\mathcal{P}g$ and $(I-\mathcal{P})g$ of $g.$ This fact, together with the validity of the corresponding results for Schr\"odinger, strongly suggests a positive answer to our question of whether Theorem~\ref{thm:Lamé_1d} extends to higher dimensions and indeed it is confirmed by the following result proved in this paper.

\begin{theorem}\label{1_Lamé}
	Let $d\geq 2$ and let $0<\gamma\leq 1/2$ if $d=2$ and $0\leq \gamma\leq 1/2$ if  $d\geq 3.$ Then any eigenvalue $z \in \C \setminus [0, \infty)$ of the perturbed Lamé operator $-\Delta^\ast + V$ satisfies
	\begin{equation}\label{bound_eigenvalues-Lamé}
		\abs{z}^\gamma \leq C_{\gamma, d, \lambda, \mu} \norm{V}_{L^{\gamma + \frac{d}{2}}(\R^d)}^{\gamma+\frac{d}{2}},
	\end{equation}
	with a constant $C_{\gamma, d, \lambda, \mu}$ independent of $V.$
\end{theorem}

Incidentally, observe that if $d\geq 3$ and $\gamma=0,$ the previous theorem provides a sufficient condition which guarantees absence of \emph{discrete} eigenvalues of $-\Delta^\ast + V.$ More specifically, if
\begin{equation}\label{smallness_latter_work}
	C_{0,d,\lambda,\mu} \norm{V}_{L^{\frac{d}{2}}(\R^d)}^{\frac{d}{2}}<1,
\end{equation} 
then the discrete spectrum $\sigma_d(-\Delta^\ast + V)$ is empty. In comparison with this result, we should mention a prior work of the present author~\cite{Lucrezia} also related to the problem of establishing sufficient conditions which disprove presence of eigenvalues. In~\cite{Lucrezia}, with a completely different approach based on the multipliers method as previously applied to Schr\"odinger operators by Fanelli, Krej\v{c}i\v{r}\'ik and Vega in~\cite{F_K_V}, \emph{total} absence of eigenvalues, \emph{i.e}, absence of both \emph{discrete} and \emph{embedded} eigenvalues, of $-\Delta^\ast + V$ was proved under the following Hardy-type subordination 
\begin{equation}\label{smallness_prior_work}
	\int_{\R^d} \abs{x}^2 \abs{V(x)}^2 \abs{u}^2\, dx\leq \Lambda^2\int_{\R^d} \abs{\nabla u}^2\, dx, \qquad \forall\, u \in [H^1(\R^d)]^d, 
\end{equation}
where $\Lambda$ is a suitable small constant (see condition (4) in~\cite{Lucrezia}).

As~\eqref{smallness_latter_work}, condition~\eqref{smallness_prior_work} is intrinsically a smallness condition, it is true, on the other hand it is satisfied by potentials with quite rough local singularities, \emph{e.g.}, $\abs{x}^{-2},$ which, instead, are ruled out by the $L^p$-type condition in Theorem~\ref{1_Lamé}.

\medskip
In attempt of including potentials with local stronger singularities, such as inverse-square type, we generalize Theorem~\ref{1_Lamé} by measuring the size of the potential in~\eqref{bound_eigenvalues-Lamé} in terms of less restrictive norms. Specifically, as first generalization, we consider potentials in the Morrey-Campanato class $\mathcal{L}^{\alpha,p}(\R^d)$ which is defined for $\alpha>0$ and $1\leq p\leq d/\alpha$ by
\begin{equation*}
	\norm{V}_{\mathcal{L}^{\alpha,p}(\R^d)}:=\sup_{x,r} r^\alpha \Big( r^{-d} \int_{B_r(x)} \abs{V(x)}^p\,dx \Big)^{\frac{1}{p}} <\infty.
\end{equation*}
In passing, notice that $1/\abs{x}^\alpha\in L^{d/\alpha,\infty}(\R^d)\subset \mathcal{L}^{\alpha, p},$ for $\alpha>0$ (we emphasize particularly the case $\alpha=2$) and $1\leq p<d/\alpha,$ however $1/\abs{x}^\alpha\notin L^{d/\alpha}=\mathcal{L}^{\alpha, d/\alpha}.$

More precisely we shall prove the following theorem.
\begin{theorem}\label{2_Lamé}
	Let $d\geq 2$ and let $(d-1)(2\gamma + d)/ 2(d-2\gamma)<p\leq \gamma + {d/2}$ with $0<\gamma\leq 1/2$ if $d=2$ and $0\leq \gamma\leq 1/2$ if $d\geq 3.$ Then any eigenvalue $z \in \C \setminus [0, \infty)$ of the perturbed Lamé operator $-\Delta^\ast + V$ satisfies
	\begin{equation}\label{bound_eigenvalues_MC-Lamé}
		\abs{z}^\gamma \leq C_{\gamma, p, d, \lambda, \mu} \norm{V}_{\mathcal{L}^{\alpha,p}(\R^d)}^{\gamma + \frac{d}{2}},
	\end{equation}
	with $\alpha=2d/(2\gamma+d)$ and a constant $C_{\gamma,p,d,\lambda, \mu}$ independent of $V.$ 
	\end{theorem}
	
	As a byproduct, in higher dimensions, the previous theorem provides a sufficient condition to guarantee absence of discrete eigenvalues. More precisely, the following corollary is immediate consequence of Theorem~\ref{2_Lamé}.
\begin{corollary}
	Let $d\geq 3,$ $(d-1)/2<p\leq d/2$ and assume
	\begin{equation*}
		C_{0, p, d, \lambda, \mu} \norm{V}_{\mathcal{L}^{2,p}(\R^d)}^{\frac{d}{2}}<1,
		\end{equation*}
		with $C_{0, p, d, \lambda, \mu}$ as in Theorem~\ref{2_Lamé} when $\gamma=0.$
				Then the perturbed Lamé operator $-\Delta^\ast + V$ has no eigenvalue in $\C\setminus [0,\infty).$
\end{corollary}

Observe that Theorem~\ref{2_Lamé} does extend Theorem~\ref{1_Lamé}, indeed from 
\begin{equation*}
	L^{\gamma + \frac{d}{2}}(\R^d)= \mathcal{L}^{\frac{2d}{2\gamma + d}, \gamma + \frac{d}{2}}(\R^d)\subseteq \mathcal{L}^{\frac{2d}{2\gamma+d},p}(\R^d), 
\end{equation*}
which holds true for $1\leq p\leq \gamma + d/2,$ in particular it follows that
\begin{equation*}
\norm{V}_{\mathcal{L}^{\frac{2d}{2\gamma+d},p}(\R^d)}\leq C_{\gamma, p, d} \norm{V}_{L^{\gamma + \frac{d}{2}}(\R^d)},
\end{equation*} 
which immediately gives~\eqref{bound_eigenvalues-Lamé} as a consequence of~\eqref{bound_eigenvalues_MC-Lamé}.  

\smallskip
Actually our investigation goes even further providing eigenvalues bounds of the type~\eqref{bound_eigenvalues-Lamé} in terms of potentials belonging to the Kerman-Saywer space $\mathcal{KS}_\alpha(\R^d),$ which is defined for $0<\alpha<d$ by
\begin{equation*}
	\norm{V}_{\mathcal{KS}_\alpha(\R^d)}:= \sup_{Q} \Big( \int_{Q} \abs{V(x)}\, dx \Big)^{-1} \int_Q \int_Q \frac{\abs{V(x)} \abs{V(y)}}{\abs{x-y}^{d-\alpha}}\, dx\, dy<\infty,
\end{equation*}   
where the supremum is taken over all dyadic cubes $Q$ in $\R^d.$ 

As $\mathcal{L}^{\alpha, p}(\R^d)\subset \mathcal{KS}_{\alpha}(\R^d)$ if $p\neq 1$ (see Section 2 in~\cite{B_B_R_V}) it is true that the Kerman-Sayer class is wider than the Morrey-Campanato class. On the other hand, it turns out that assuming solely $V\in \mathcal{KS}_\alpha$ is not enough to get bound~\eqref{bound_eigenvalues_MC-Lamé} with $\norm{V}_{\mathcal{L}^{\alpha,p}(\R^d)}$ replaced by $\norm{V}_{\mathcal{KS}_\alpha(\R^d)}.$  Indeed, additionally, we will ask the potential $V$ to belong to the Muckenhoupt $A_2(\R^d)$ class of weights which is defined, in general, for $1<p<\infty$ as the set of measurable non-negative function $w$ such that
\begin{equation}\label{Q_p}
	Q_p(w):= \sup_Q \Bigg ( \frac{1}{\abs{Q}} \int_Q w(x)\, dx \Bigg) \Bigg ( \frac{1}{\abs{Q}} \int_Q w(x)^{- \frac{1}{p-1}}\, dx \Bigg)^{p-1}\leq C,
\end{equation}
where $Q$ is any cube in $\R^d$ and $C$ is a constant independent of $Q.$

More precisely, we shall prove the following result.
\begin{theorem}\label{KS_3_Lamé}
	Let $d\geq 2$ and let $1/3\leq \gamma<1/2$ if $d=2$ and $0\leq \gamma<1/2$ if $d\geq 3.$ If $V\in A_2(\R^d),$ then any eigenvalue $z\in \C\setminus [0, \infty)$ of the perturbed Lamé operator $-\Delta^\ast + V$ satisfies
	\begin{equation}\label{bound_eigenvalues_KS-Lamé}
		\abs{z}^\gamma \leq C_{\gamma,d,\lambda,\mu} \norm{V^\beta}_{\mathcal{KS}_{\alpha}}^{\frac{1}{\beta}(\gamma +\frac{d}{2})},
	\end{equation}
	with $\alpha=\frac{2d}{2\gamma +d} \beta$ and $\beta=\frac{(d+2\gamma)(d-1)}{2(d-2\gamma)}$ and  a constant $C_{\gamma,d,\lambda, \mu}$ independent of $V.$ 
\end{theorem}	

As a consequence of the previous result, one gets the following corollary on absence of discrete eigenvalues.
\begin{corollary} 
	Let $d\geq 3$ and assume
	\begin{equation*}
		C_{0, d,\lambda,\mu}\norm{V^\frac{d-1}{2}}_{\mathcal{KS}_{d-1}}^\frac{d}{d-1}<1,
	\end{equation*}
	with $C_{0, d, \lambda, \mu}$ as in Theorem~\ref{KS_3_Lamé} when $\gamma=0.$
	Then the perturbed Lamé operator $-\Delta^\ast + V$ has no eigenvalue in $\C\setminus [0, \infty).$
\end{corollary}

It is worth comparing Theorem~\eqref{KS_3_Lamé} with the analogous result in~\cite{Lee_Seo} (\emph{cfr.} Theorem 1.1) for Schr\"odinger operators. Here, the bound~\eqref{bound_eigenvalues_KS-Lamé} was obtained without the additional assumption  $V\in A_2(\R^d),$ then showing a peculiar feature of the Lamé operator.    

Roughly speaking, the philosophy is that thanks to the Helmholtz decomposition which, as shown in~\eqref{favorable_form}, makes the resolvent operator $(-\Delta^\ast -z)^{-1}$ ``behave'' like a sum of two resolvent $(-\Delta-z)^{-1},$ at first we can perform our analysis at the level of the much more investigated Schr\"odinger operators, estimating the two components in~\eqref{favorable_form} separately. Then, in order to get bound~\eqref{bound_eigenvalues_MC-Lamé} and~\eqref{bound_eigenvalues_KS-Lamé}, respectively, these two pieces have to be recombined together on weighted $L^2$-spaces and it is at this step that the $A_2$ assumption comes into play. We refer to Section~\ref{Sec:Proofs} for greater details.

Marginally, observe that, even though the aforementioned strategy underpins the proof of both~\eqref{bound_eigenvalues_MC-Lamé} and~\eqref{bound_eigenvalues_KS-Lamé}, good property of the $\mathcal{L}^{\alpha, p}$ class allowed to drop the $A_2$ assumption in Theorem~\ref{2_Lamé} (see Lemma~\ref{good_property_MC} below).   

\medskip
It is very well known fact that distinguishing whether $-\Delta + V$ has  finite (possibly empty) or infinite discrete spectrum depends on the large $x$ fall-off of the potential. More specifically, it is mainly consequence of the uncertainty principle, quantified by the Hardy inequality
\begin{equation*}
	-\Delta \geq  \frac{(d-2)^2}{4} \frac{1}{~\abs{x}^2}
\end{equation*}
 that the borderline is marked by an inverse-square type behavior at infinity.     

In scattering theory and in particular in matter of determining existence of wave operators, again the large $x$ behavior of the potential plays a central role, in this context the threshold is given by a Coulomb-type asymptotic decay and positive results require $\abs{x}^{-\alpha}$ with $\alpha>1.$ Observe that our previous theorems essentially restrict to $\abs{x}^{-\alpha}$ decay with $\alpha> 2d/d+1$ and therefore are not fully satisfactory in the perspective of their possible application to stationary scattering theory. The gap is filled in the following result by using a suitable interpolation argument.

\begin{theorem}\label{4_Lamé}
	Let $d\geq 2,$ $\gamma>\frac{1}{2}$ and $\alpha>\gamma -\frac{1}{2}.$ Then any eigenvalue $z\in \C\setminus [0,\infty)$ of the perturbed Lamé operator $-\Delta^\ast + V$ satisfies
	\begin{equation}\label{bound_eigenvalues_4}
		\abs{z}^\gamma \leq C_{\gamma,\alpha,d, \lambda, \mu} \norm{V}_{L^q(\langle x \rangle^{2\alpha}\, dx)}^q,
	\end{equation}
	with $q=2\gamma + (d-1)/2$ and a constant $C_{\gamma, \alpha, d, \lambda, \mu}$ independent of $V.$
\end{theorem}
Here we used the notation $\langle x\rangle:= (1+ \abs{x}^2)^{1/2},$ moreover, given a measurable function $w,$ $L^p(w dx)$ stands for the $w$-weighted $L^p$ space on $\R^d$ with measure $w(x) dx.$

\bigskip
The rest of the paper is organized as follows: In Section 2, in attempt of making the paper sufficiently self-contained, we collect some preliminary facts on the Helmholtz decomposition. Here, some properties of Lamé operator are provided too. Among them, particular emphasis will be given to uniform estimates for the resolvent operator $(-\Delta^\ast -z)^{-1}$ which will represent the main ingredient in the proof of our aforementioned results.

Contrarily to the much more investigated Schr\"odinger operator, up to our knowledge, eigenvalue bounds of the form~\eqref{real-valued-V-bound} for the perturbed Lamé operator are unknown even in the self-adjoint situation. Although in this case the proof of~\eqref{real-valued-V-bound} follows almost verbatim the instead-well-known one for Schr\"odinger, we decided to dedicate Section 3 to prove it anyhow. The advantage of this choice comes out in the possibility of explicitly showing the deep differences and difficulties that arise passing from the self-adjoint to the non self-adjoint framework which, instead, is fully analyzed in Section 4. In particular, Section 4 is devoted to the prood of the main results stated in the introduction.

\subsection*{Notations.}

In this paper both scalar and vector-valued functions are considered. In attempt of produce no confusion, we clarify here that notation like $f, g, u$ are reserved for vector-fields, instead $\phi, \psi$ are set aside for scalar functions. 

\medskip
When the letter adopted for denoting a vector field contains already a subscript, the standard subscript notation $\cdot_j$ for the $j$-th component will be replaced by the superscript $\cdot^{(j)},$ \emph{e.g}, the $j$-th component of the vector field $u_S$ is indicated by $u_S^{(j)}.$ 

\medskip
We use the following definition for the $L^p$- norm of a vector field $u\in [L^p(\R^d)]^d:$
\begin{equation*}
	\norm{u}_{[L^p(\R^d)]^d}:= \Big (\sum_{j=1}^d \norm{u_j}_{L^p(\R^d)}^p \Big)^\frac{1}{p}.
\end{equation*}

\medskip
In order to lighten the presentation, in the following we usually abbreviate both $\norm{\cdot}_{L^p(\R^d)}$ and $\norm{\cdot}_{[L^p(\R^d)]^d},$ the $L^p$-norm of scalar and vector-valued functions, respectively, with $\norm{\cdot}_p.$ This, in general, could create some ambiguity, on the other hand, it will be clear from the context and the notation used there if $\norm{\cdot}_p$ stands for one or the other norm.

\medskip
Again, the notation $\langle \cdot, \cdot \rangle$ will denote both $\langle \cdot, \cdot \rangle_{L^2(\R^d)}$ and $\langle \cdot, \cdot \rangle_{[L^2(\R^d)]^d},$ where the latter extends in an obvious way the usual definition of the former, namely given $f, g \in [L^2(\R^d)]^d,$ one defines 
\begin{equation*}
	\langle f, g \rangle_{[L^2(\R^d)]^d}:= \sum_{j=1}^d \langle f_j, g_j \rangle_{L^2(\R^d)}.
\end{equation*}

\medskip
Let $E$ and $F$ be two Banach spaces and let $T\colon E \to F$ be a bounded linear operator from $E$ into $F.$  The notation $\norm{T}_{E\to F}$ will be used to denote the operator norm of $T.$

\subsection*{Acknowledgment}
The author is deeply grateful to Prof. L. Fanelli for bringing the attention to the setting of the problem and for valuable comments on the draft of the manuscript. The author also thanks Prof. D. Krej\v{c}i\v{r}\'ik for useful discussions which greatly enrich the  paper. The author is greatful to Prof. M. Correggi for helpful references' suggestions. The author thanks   
the hospitality of \emph{Sapienza University of Rome} where this work was initiated and \emph{Czech Technical University in Prague} where it was developed in part.
The author gratefully acknowledges financial support by the Deutsche Forschungsgemeinschaft (DFG) through CRC 1173.

\section{Preliminaries}
This section is concerned with recalling some properties connected with the Helmholtz decomposition together with  stating and proving some related consequences on Lamé operators that will strongly enter the proof of our results later. If the first part wants to be just a remainder of very well known results on the Helmholtz decomposition and therefore can be safely skipped by any reader already familiar with this topic, the subsequent subsections, namely Subsection~\ref{ss:resolvent_repr} and Subsection~\ref{u_r_e}, represent and important part of the paper. More specifically, it is there that uniform resolvent estimates for the resolvent operator $(-\Delta^\ast-z)^{-1}$ are provided which are important in their own right.

\subsection{Helmholtz decomposition}\label{Ssec:Helmholtz}
\begin{theorem}[Helmholtz decomposition]\label{Helmholtz_decomposition}
	Let $d\geq 2$ and let $\Omega\subseteq \R^d$ be either an open, bounded, simply connected, Lipschitz domain or the entire $\R^d.$ Then any square-integrable vector field  $f= (f^{(1)}, f^{(2)}, \dots, f^{(d)}) \in [L^2(\Omega)]^d$ can be uniquely decomposed as
	\begin{equation*}
		f=f_S+ f_P,
	\end{equation*}
	where $f_S$ is a divergence-free vector field with null normal derivative and $f_P$ is a gradient. Moreover the two components are orthogonal in a $L^2$- sense. Specifically, the Pythagorean identity
	\begin{equation}\label{Pythagorean_L2}
		\norm{f}_{[L^2(\Omega)]^d}^2= \norm{f_S}_{[L^2(\Omega)]^d}^2 + \norm{f_P}_{[L^2(\Omega)]^d}^2
	\end{equation}
	holds true.
\end{theorem} 
\begin{proof}
	Though the proof simply relies on classical techniques, for sake of completeness we provide a brief sketch of the proof.
	
	Let $f\in [L^2(\Omega)]^d$ and let $\nu$ denote the unit outward normal vector at the boundary $\partial \Omega.$
	We consider the trivial decomposition $f=f-\nabla \psi + \nabla \psi.$ In order for $f-\nabla \psi$ to be the divergence-free component $f_S$ of $f$ with null normal derivative, $\psi$ must satisfy the boundary value problem
	\begin{equation}\label{Poisson_Neumann}
		\begin{system}
			\Delta \psi&= \div f  \quad &\text{in}\, \Omega\\
			\frac{\partial \psi}{\partial \nu}&= f\cdot \nu &\text{on}\, \partial \Omega.
		\end{system}
	\end{equation}
	In passing, observe that~\eqref{Poisson_Neumann} is a Poisson problem with Neumann boundary conditions, therefore it admits a unique solution $\psi\in H^1(\Omega)$ modulo additive constants.
	
	Let us now prove the uniqueness of the decomposition. Let $f_1, f_2 \in [L^2(\Omega)]^d$ and let $\psi_1, \psi_2\in H^1(\Omega)$ be such that $\div f_1=\div f_2=0$ in $\Omega$ and $f_1\cdot \nu=f_2\cdot \nu=0$ on $\partial \Omega$ and with the property that $f=f_1- \nabla \psi_1= f_2-\nabla \psi_2$ which gives
	\begin{equation*}
		f_1-f_2=\nabla(\psi_1-\psi_2).
	\end{equation*}
	Multiplying the latter by $f_1-f_2,$ integrating over $\Omega$ and integrating by parts, one has
	\begin{equation*}
		\int_\Omega \abs{f_1-f_2}^2\, dx=\int_\Omega (f_1-f_2)\cdot \nabla (\psi_1-\psi_2)\, dx=-\int_\Omega \div (f_1-f_2) (\psi_1-\psi_2)\, dx=0.
	\end{equation*}
	It follows that $f_1=f_2$ and therefore $\psi_1=\psi_2$ up to an additive constant. To sum up, we have proved that $f$ can be uniquely written as a sum of a divergence free vector field $f_S$ with null normal derivative on the boundary $\partial \Omega$ and a gradient $f_P,$ where an explicit expression for $f_S$ and $f_P$ is provided by the former construction, specifically
	\begin{equation}\label{explicit_components}
		f_S:= f-\nabla \psi, \qquad f_P:=\nabla \psi,
	\end{equation}
	with $\psi\in H^1(\Omega)$ unique (up to additive constant) solution of~\eqref{Poisson_Neumann}. 
	
	At last, observe that the $L^2$-orthogonality of $f_S$ and $f_P$ and in particular identity~\eqref{Pythagorean_L2} are immediate consequences of the properties of the two components.
	This yields the proof.
\end{proof}
\begin{remark}
	Notice that, as soon as the vector field $f$ is more regular, for instance, say $f\in [H^1(\Omega)]^d$ as it suits our purposes, then the components $f_S$ and $f_P$ of the Helmholtz decomposition are even $H^1$- orthogonal and in particular
	\begin{equation*}
		\norm{\nabla f}_{[L^2(\Omega)]^d}^2= \norm{\nabla f_S}_{[L^2(\Omega)]^d}^2 + \norm{\nabla f_P}_{[L^2(\Omega)]^d}^2. 
	\end{equation*}
\end{remark}

\medskip
Now we are in position to provide the rigorous definition of the so-called Leray projection operator already mentioned in the introduction.
\begin{definition}\label{def:Leray_projector}
	Consider the setting of Theorem~\ref{Helmholtz_decomposition}. Let $\mathcal{P}$ be the orthogonal projection of $[L^2(\Omega)]^d$ into the subspace of divergence-free vector fields with null normal derivative on $\partial \Omega.$ Then $\mathcal{P}$ is called Leray projection operator. More precisely, for any $f\in [L^2(\Omega)]^d,$ it holds that
	\begin{equation*}
		\mathcal{P}f=\mathcal{P}(f_S+ f_P)=f_S,
	\end{equation*}
	where $f=f_S+f_P$ is the Helmholtz decomposition of $f$ as constructed in the previous theorem.
\end{definition}

The following result is easily proven.
\begin{lemma}\label{boundedness_Leray}
	Consider the setting of Theorem~\ref{Helmholtz_decomposition}. The Leray projection operator $\mathcal{P}\colon [L^2(\Omega)]^d\to [L^2(\Omega)]^d$ is bounded.
\end{lemma}
\begin{proof}
	By definition, for any $f\in [L^2(\Omega)]^d,$ $\mathcal{P}f:= f-\nabla \psi,$ where $\psi\in H^1(\Omega)$ is the unique solution to the boundary value problem~\eqref{Poisson_Neumann}. In particular, $\psi$ achieves $\int_\Omega \abs{\nabla \psi - f}^2\, dx= \inf_{\phi\in H^1(\Omega)} J(\phi),$ where
	\begin{equation*}	
		J(\phi):=\int_\Omega \abs{\nabla \phi -f}^2\, dx.
	\end{equation*}
	Now, from $\inf_{\phi\in H^1(\Omega)} J(\phi)\leq \int_\Omega \abs{f}^2\, dx,$ we immediately get $\norm{\mathcal{P}f}_{[L^2(\Omega)]^d}\leq \norm{f}_{[L^2(\Omega)]^d},$ which is the thesis.
\end{proof}

\medskip
As showed by the following result, in the specific case $\Omega=\R^d,$ the Leray projection operator has a favorable form in terms of Riesz transform $\mathcal{R}=(\mathcal{R}_1, \mathcal{R}_2, \dots, \mathcal{R}_d)$ defined for any $\phi\in L^2(\R^d),$ in Fourier space, by
  \begin{equation}\label{Riesz_transform}
		\widehat{\mathcal{R}_j \phi}(\xi)= -i \frac{\xi_j}{\abs{\xi}} \widehat{\phi}(\xi), \qquad j=1,2,\dots, d.
	\end{equation}
	
\begin{lemma}
	Consider the setting of Definition~\ref{def:Leray_projector}, fix $\Omega=\R^d$ and let $f\in [L^2(\R^d)]^d$ be any square-integrable vector field on $\R^d.$ Then the $j$-th component of $\mathcal{P}f$ can be written as
	\begin{equation}\label{representation_Leray}
		(\mathcal{P}f)_j= f_j + \sum_{k=1}^d \mathcal{R}_j \mathcal{R}_k f_k,
	\end{equation}
	for any $j=1,2,\dots, d$ and where $\mathcal{R}$ denote the Riesz transform defined in~\eqref{Riesz_transform}. 
\end{lemma}
\begin{proof}
	Let $f\in [L^2(\R^d)]^d.$ It follows from Theorem~\ref{Helmholtz_decomposition} and Definition~\ref{def:Leray_projector} that $\mathcal{P}f=f_S= f-\nabla \psi,$ where $\psi$ satisfies $\Delta \psi= \div f.$ In Fourier space this yields
	\begin{equation*}
		\widehat{\psi}(\xi)= - i\sum_{k=1}^d \frac{\xi_k}{\abs{\xi}^2} \widehat{f_k}(\xi).
	\end{equation*}
	In particular, for any $j=1,2,\dots, d,$ this gives
	\begin{equation*}
		\widehat{\partial_j \psi}(\xi):=i \xi_j \widehat{\psi}(\xi)=\sum_{k=1}^d \frac{\xi_j \xi_k}{\abs{\xi}^2} \widehat{f_k}(\xi)
	\end{equation*}
	and so, by using the Fourier representation of the Riesz transform~\eqref{Riesz_transform}, 
	\begin{equation}\label{nabla-psi}
		\partial_j \psi=- \sum_{k=1}^d \mathcal{R}_j \mathcal{R}_k f_k,
	\end{equation}
	for any $j=1,2,\dots, d.$
	From the latter, one immediately gets~\eqref{representation_Leray} and this concludes the proof.
\end{proof}

\medskip
In passing, observe that since $\psi$ is a solution of~\eqref{Poisson_Neumann}, then it follows easily by elliptic estimates that 
\begin{equation}\label{starting_boundedness}
	\norm{\nabla \psi}_{[L^2(\Omega)]^d}\leq \norm{f}_{[L^2(\Omega)]^d}
\end{equation} 
(the same would follow from the boundedness of the Leray operator $\mathcal{P}$ (see Lemma~\ref{boundedness_Leray}) but in this case we would get a worse bound).

Now, we want to show that as a consequence of the representation~\eqref{nabla-psi} of $\nabla \psi$ in terms of the Riesz transform, we can get boundedness of type~\eqref{starting_boundedness} replacing the $L^2$-norm with suitable $L^p$-\emph{weighted} norms.  

Notice that, in general, proving boundedness in weighted $L^p$-space does not come as a mere consequence of elliptic estimates as for~\eqref{starting_boundedness} and, in fact, does require a more involved analysis. In our case, it will follow from special well known properties of the Riesz transform that we summarize in the following lemma.

\begin{lemma}[Boundedness Riesz transform]
	\label{boundedness-Riesz-transform}
	Let $1<p<\infty$ and $p'$ such that $1/p+1/p'=1$ and let $w$ be a weight in the  $A_p(\R^d)$-class (see definition~\eqref{Q_p}). Then, for any $j=1,2,\dots, d,$ the following bounds on the operator norms of the Riesz transform $\mathcal{R}_j$ hold true:
	\begin{equation}\label{Riesz_1}
		\norm{\mathcal{R}_j}_{L^p\to L^p}=\cot\Big(\frac{\pi}{2p^\ast} \Big)=:c_p, \qquad p^\ast:=\max\{p,p'\},  
	\end{equation}
	\begin{equation}\label{Riesz_2}
		\norm{\mathcal{R}_j}_{L^p(w)\to L^p(w)}\leq c_{p,d}\, Q_p(w)^r, \qquad r:= \max\{1, p'/p\}.
	\end{equation}
	Moreover, both the bound are sharp, \emph{i.e.} the best possible bound is established.
\end{lemma}
\begin{proof}
	The proof of the sharp bound~\eqref{Riesz_1} can be found in~\cite{B_W} (see also~\cite{Ca_Zy}), inequality~\eqref{Riesz_2} can be found in~\cite{Petermichl} (see also~\cite{C_F}).
\end{proof}

Now we are in position to state and prove the aforementioned boundedness of the operator $\nabla \psi.$ More precisely, we are interested in proving the following result.
\begin{lemma}\label{lemma:orthogonality}
	Consider $\psi$ the unique solution to 
	\begin{equation}\label{Poisson}
		\Delta \psi= \div u.
	\end{equation}
	Then, for any $1<p<\infty,$ the following estimates hold true:
	\begin{gather}
		\label{bound_nabla-psi}
			\norm{\nabla \psi}_{L^p}\leq c_p^2 \,d \norm{u}_{L^p},\\
		\label{w_bound_nabla-psi}
			\norm{\nabla \psi}_{L^p(w)}\leq c_{p,d}^2\, Q_p(w)^{2r} \, d \norm{u}_{L^p(w)},
	\end{gather}
	where $w, c_p, c_{p,d}$ and $r$ are as in Lemma~\ref{boundedness-Riesz-transform}.
\end{lemma}
\begin{proof}
	We will prove only~\eqref{bound_nabla-psi}, the proof of~\eqref{w_bound_nabla-psi} is analogous.
	
	Using the representation~\eqref{nabla-psi}, the bound for the Riesz operator~\eqref{Riesz_1} and the H\"older inequality for discrete measures, we get
	\begin{equation*}
		\begin{split}
			\norm{\nabla \psi}_{L^p}&= \Big( \sum_{j=1}^d \norm{\partial_j \psi}_{L^p}^p \Big)^\frac{1}{p}
			\leq \Bigg( \sum_{j=1}^d \Big(\sum_{k=1}^d \norm{\mathcal{R}_j\mathcal{R}_k u_k}_{L^p} \Big)^p \Bigg)^\frac{1}{p}\\
			&= c_p^2\, d^\frac{1}{p} \sum_{k=1}^d \norm{u_k}_{L^p}\\
			&\leq c_p^2\, d\, \norm{u}_{L^p}.
		\end{split}
	\end{equation*}
	This gives~\eqref{bound_nabla-psi}. Bound~\eqref{w_bound_nabla-psi} follows in the same way using~\eqref{Riesz_2}.
	\end{proof}

As a consequence of the previous lemma, we are able to prove the following (almost-) orthogonality result that we will strongly use in the future.
\begin{lemma}\label{lemma:almost_orthogonality}
	Let $u=u_S + u_P$ be the Helmholtz decomposition of $u.$ Then, for $1<p<\infty,$ the following estimates hold true:
	\begin{gather}
		\label{orth}
			\norm{u_S}_{L^p} + \norm{u_P}_{L^p}\leq (1+ 2c_p^2 \,d) \norm{u}_{L^p},\\
		\label{w_orth}
			\norm{u_S}_{L^p(w)} + \norm{u_S}_{L^p(w)}\leq (1+ 2c_{p,d}^2\, Q_p(w)^{2r} \, d )\norm{u}_{L^p(w)},
	\end{gather}
	where $w, c_p, c_{p,d}$ and $r$ are as in Lemma~\ref{boundedness-Riesz-transform}.
\end{lemma}
\begin{proof}
	The proof is a direct consequence of Lemma~\ref{lemma:orthogonality}. We know from Theorem~\ref{Helmholtz_decomposition} (see~\eqref{explicit_components}) that 
	\begin{equation*}
		u_S= u- \nabla \psi,\qquad u_P=\nabla \psi,
	\end{equation*}
	with $\psi$ the unique solution to the Poisson problem~\eqref{Poisson}. Then, it is easy to see that estimate~\eqref{orth} holds true, indeed
	\begin{equation*}
		\norm{u_S}_{L^p} + \norm{u_P}_{L^p}\leq \norm{u}_{L^p} + 2 \norm{\nabla \psi}_{L^p}\leq (1+ 2 c_p^2\, d) \norm{u}_{L^p},
	\end{equation*}
	where in the last inequality we used~\eqref{bound_nabla-psi}. As it is analogue, we skip the proof of~\eqref{w_orth}.
\end{proof} 

\subsection{A favorable representation for \texorpdfstring{$(-\Delta^\ast -z)^{-1}$}{TEXT}}
\label{ss:resolvent_repr}
As already mentioned in the introduction, the usage, as a starting point in our proofs, of an adaptation of the Birman-Schwinger principle to our elasticity context, requires a better understanding of the action of the resolvent operator $(-\Delta^\ast -z)^{-1},$ well-defined for any $z\in \C\setminus [0,\infty),$ associated with the Lamé operator.

The following easy consequence of Helmholtz decomposition will be useful to this end.
		\begin{lemma}
			Let $d\geq 2$ and let $f$ be a suitably smooth vector field sufficiently rapidly decaying at infinity. Then $-\Delta^\ast$ acts on $f=f_S+f_P$ as 
			\begin{equation}\label{simple_writing}
				-\Delta^\ast f= -\mu \Delta f_S -(\lambda +2 \mu) \Delta f_P,
			\end{equation}   
			 where  $f_S$ is a divergence free vector field and $f_P$ a gradient.		
			\end{lemma}
		
		Now we are in position to show the validity of identity~\eqref{favorable_form}, stated in the introduction, which follows as a consequence of~\eqref{simple_writing} together with the $H^1$-orthogonality of the components of the Helmholtz decomposition. This is object of the following lemma.
		\begin{lemma}
		\label{lemma:resolvent_Lamé}
		Let $z\in \C\setminus [0, \infty)$ and $ g\in [L^2(\R^d)]^d.$ Then the identity 
		\begin{equation}\label{resolvent}
			(-\Delta^\ast -z)^{-1} g= \frac{1}{\mu} \big(-\Delta - \tfrac{z}{\mu}\big)^{-1} g_S + \frac{1}{\lambda + 2\mu} \big(-\Delta - \tfrac{z}{\lambda + 2\mu}\big)^{-1} g_P
		\end{equation}
		holds true, where $g=g_S+ g_P$ is the Helmholtz decomposition of $g.$
		\end{lemma}
		\begin{proof}
			Given $g\in [L^2(\R^d)]^d,$ we want to obtain an explicit expression of the vector field $f,$ defined by
		\begin{equation}\label{f}
			f:=(-\Delta^\ast -z)^{-1} g.
		\end{equation}
		Observe that, since $z\notin \sigma(-\Delta^\ast)=[0,\infty),$ the previous is equivalent to
		\begin{equation*}
			(-\Delta^\ast -z) f=g.
		\end{equation*}
		Now, writing the Helmholtz decomposition of $f$ and $g,$ namely $f=f_S + f_P$ and $g=g_S + g_P,$ respectively, and using~\eqref{simple_writing}, the previous can be re-written as
		\begin{equation*}
			-\mu \Delta f_S - (\lambda + 2\mu)\Delta f_P - z f_S -z f_P= g_S + g_P.
		\end{equation*}
		The $H^1$- orthogonality of the two components of the decomposition enables us to split this intertwining equation for both the two components into a system of two decoupled equations, i.e
			\begin{equation*}
				\begin{system}
					-\mu \Delta f_S -z f_S =g_S,\\
					-(\lambda + 2\mu) \Delta f_P -z f_P =g_P,
				\end{system}
			\end{equation*}
			or, equivalently,
			\begin{equation*}
				\begin{system}
					\mu \big(-\Delta -\frac{z}{\mu} \big) f_S =g_S,\\
					(\lambda + 2\mu) \big( -\Delta -\frac{z}{\lambda + 2\mu} \big) f_P =g_P.
				\end{system}
			\end{equation*}
			Hypothesis~\eqref{Lamé_parameter} ensures that $\frac{z}{\mu}, \frac{z}{\lambda + 2\mu}\notin \sigma(-\Delta),$ hence
			\begin{equation*}
				f_S= \frac{1}{\mu} \big( -\Delta - \tfrac{z}{\mu} \big)^{-1} g_S, \qquad f_P= \frac{1}{\lambda + 2\mu} \big( -\Delta - \tfrac{z}{\lambda + 2\mu} \big)^{-1} g_P.
			\end{equation*}
			Plugging these explicit expressions in $f=f_S + f_P,$ from~\eqref{f} we obtain~\eqref{resolvent}. 
	\end{proof}
We underline that Lemma~\ref{lemma:resolvent_Lamé} was already proved in~\cite{B_F-G_P-E_R_V}, its statement and proof were provided also here only for reader's convenience.

\subsection{Uniform resolvent estimates}
	\label{u_r_e}
	As already mentioned, in a non self-adjoint framework, the unavailability of a variational characterization of the spectrum causes that Sobolev inequalities no longer suffice to prove spectral bounds. To overcome this lack, uniform resolvent estimates have been recognized as a crucial tool   
to fruitfully address this problem. For this reason in this subsection we shall prove uniform estimate for the operator $(-\Delta^\ast -z)^{-1}$ that will be the fundamental ingredient in the proof of our main results later on. 

\medskip
Observe that from the representation~\eqref{resolvent}, it is reasonable to expect that estimates for $(-\Delta^\ast -z)^{-1}$ should follow as a consequence of the corresponding estimates (if available) for the resolvent operator $(-\Delta-z)^{-1}$ associated to the Laplacian. This is true, indeed, as we will see repeatedly in this paper, the underlying strategy to treat issues concerning the Lamé operator is, actually, to operate at first at the level of the Laplacian taking advantage of the representation~\eqref{resolvent} once the Helmholtz decomposition is operated. On the other hand, this procedure has the main drawback of producing  \emph{separated} outcomes on the \emph{single} components of the Helmholtz decomposition, which later must be recombined together in order to get meaningful results in application. This need of recombination requires providing suitable (almost-) orthogonality results for the Helmholtz components, roughly speaking some inequality like $\norm{u_S} + \norm{u_P}\lesssim \norm{u}$ (see Lemma~\ref{lemma:almost_orthogonality}), which turn out to be highly non-trivial to get and require deep result from harmonic analysis and in particular from singular integrals theory (see Lemma~\ref{boundedness-Riesz-transform}). 

The following theorem collects the main estimates for the resolvent operator $(-\Delta-z)^{-1}$ that we will use for our purpose.    

\begin{theorem}[Uniform estimates for $(-\Delta -z)^{-1}$]
	\label{thm:uniform_estimate}
	Let $z\in \C\setminus [0,\infty).$ Then the following estimates for $(-\Delta-z)^{-1}$ hold true.
	\begin{enumerate}[i)]
		\item Let $1<p\leq 6/5$ if $d=2,$ $2d/(d+2)\leq p \leq 2(d+1)/(d+3)$ if $d\geq 3$ and let $p'$ such that $1/p + 1/p'=1.$ Then
			\begin{equation}\label{res_1}
				\norm{(-\Delta-z)^{-1}}_{L^p\to L^{p'}}\leq C_{p,d}\abs{z}^{-\frac{d+2}{2} + \frac{d}{p}}.
			\end{equation}
		\item Let $\alpha>1/2.$ Then
			\begin{equation}\label{st_scattering}
				\norm{(-\Delta -z)^{-1}}_{L^2(\langle x\rangle^{2\alpha}) \to L^2(\langle x\rangle^{-2\alpha})}\leq C_{\alpha, d} \abs{z}^{-\frac{1}{2}}.
			\end{equation}
		\item Let $3/2< \alpha < 2$ if $d=2,$ $2d/(d+1)<\alpha \leq 2$ if $d\geq 3$ and let $(d-1)/2(\alpha-1)< p \leq d/\alpha.$ Then for any non-negative function $V$ in $\mathcal{L}^{\alpha, p}(\R^d)$
		\begin{equation}\label{res_2}
			\norm{(-\Delta-z)^{-1}}_{L^2(V^{-1})\to L^{2}(V)}\leq C_{\alpha, p,d} \norm{V}_{\mathcal{L}^{\alpha, p}(\R^d)}\abs{z}^{-1+ \frac{\alpha}{2}}.
		\end{equation}
		\item Let $3/2\leq \alpha < 2$ if $d=2,$ $d-1\leq \alpha < d$ if $d\geq 3$ and let $\beta=(2\alpha -d+1)/2.$ Then for any non-negative function $V$ such that $\abs{V}^\beta\in \mathcal{KS}_{\alpha}(\R^d)$
		\begin{equation}\label{res_3}
			\norm{(-\Delta-z)^{-1}}_{L^2(V^{-1})\to L^{2}(V)}\leq C_{\alpha, d} \norm{\abs{V}^\beta}_{\mathcal{KS}_{\alpha}}^\frac{1}{\beta} \abs{z}^{-\frac{\alpha - d+1}{2\alpha -d+1}}.
		\end{equation}
	\end{enumerate}
\end{theorem}
\begin{proof}
	Proof of~\eqref{res_1} can be found in the work~\cite{K_R_S} by Kenig, Ruiz and Sogge. Estimate~\eqref{st_scattering} is proved in the pioneering work by Agmon~\cite{Agmon} (Lemma 4.1 there), see also~\cite{Sylvester}. Proof of~\eqref{res_2} was provided by Frank in~\cite{Frank} (see also~\cite{Ch_Sa,Ch_Ru}). Finally, estimate~\eqref{res_3} is proved by Lee and Seo in~\cite{Lee_Seo}.
\end{proof}

Now we are in position to state the corresponding estimate for the resolvent operator $(-\Delta^\ast -z)^{-1}.$
\begin{theorem} \label{thm:unif_resolvent_Lamé}
	Let $z\in \C\setminus [0,\infty).$ Then, under the same hypotheses of Theorem~\ref{thm:uniform_estimate}, the following estimates for $(-\Delta^\ast-z)^{-1}$ hold true.
	\begin{gather}
		\label{Lamé_res_1}
				\norm{(-\Delta^\ast-z)^{-1}}_{L^p\to L^{p'}}\leq C_{p,d,\lambda, \mu}\abs{z}^{-\frac{d+2}{2} + \frac{d}{p}}.
			\\
			\label{Lamé_st_scattering}
				\norm{(-\Delta^\ast -z)^{-1}}_{L^2(\langle x\rangle^{2\alpha}) \to L^2(\langle x\rangle^{-2\alpha})}\leq C_{\alpha, d, \lambda, \mu} \abs{z}^{-\frac{1}{2}}.
			\\
			\label{Lamé_res_2}
			\norm{(-\Delta^\ast-z)^{-1}}_{L^2(V^{-1})\to L^{2}(V)}\leq C_{\alpha, p,d,\lambda, \mu} \norm{V}_{\mathcal{L}^{\alpha, p}(\R^d)}\abs{z}^{-1+ \frac{\alpha}{2}}.
		\end{gather}
		If, in addition, $V\in A_2(\R^d),$ then
		\begin{equation}\label{Lamé_res_3}
			\norm{(-\Delta^\ast-z)^{-1}}_{L^2(V^{-1})\to L^{2}(V)}\leq C_{\alpha, d,\lambda, \mu} \norm{\abs{V}^\beta}_{\mathcal{KS}_{\alpha}}^\frac{1}{\beta} \abs{z}^{-\frac{\alpha - d+1}{2\alpha -d+1}}.
		\end{equation}
\end{theorem}

\begin{remark}
	Notice that in order to prove~\eqref{Lamé_res_3}, that is  when the bound on the resolvent operator norm is measured in term of potentials in the Kerman-Saywer space, the additional assumption of $V$ belonging to the $A_2$ class of weights is required, on the contrary this auxiliary hypothesis is not needed when the estimate involves potentials in the Morrey-Campanato class (see~\eqref{Lamé_res_2}). This fact is mainly due to a good behavior of functions in the Morrey-Campanato space in relation with the $A_p$ class of weights. This property is clarified in the following result (refer to~\cite{Ch_Fr}, Lemma 1).    
\end{remark}
\begin{lemma}\label{good_property_MC}
	Let $V$ be a non-negative function in $\mathcal{L}^{\alpha, p}(\R^d)$ with $0<\alpha<d$ and $1<p\leq d/\alpha.$ If $r$ is such that $1<r<p,$ then $W:= (M V^r)^{1/r}\in A_1(\R^d)\cap \mathcal{L}^{\alpha,p}(\R^d),$ where $M$ denotes the usual Hardy-Littlewood maximal operator. Furthermore there exists a constant $c$ independent of $V$ such that 
	\begin{equation*}
		\norm{W}_{\mathcal{L}^{\alpha, p}}\leq c \norm{V}_{\mathcal{L}^{\alpha, p}}.
	\end{equation*}
	In passing, notice that $V(x)\leq W(x)$ for almost every $x\in \R^d.$
\end{lemma} 
Now we are in position to prove Theorem~\ref{thm:unif_resolvent_Lamé}.

\medskip
\begin{proof}[Proof of Theorem~\ref{thm:unif_resolvent_Lamé}]
	As already mentioned, the proof will basically rely on the interplay between the favorable representation~\eqref{resolvent} of $(-\Delta^\ast -z)^{-1}$ in terms of the resolvent of the Laplace operator together with the uniform estimates for $(-\Delta-z)^{-1}$ summarized in Theorem~\ref{thm:uniform_estimate} and the orthogonality result stated in Lemma~\ref{lemma:almost_orthogonality}.
	
	Let us first consider the proof of~\eqref{Lamé_res_1}.
	
	Let $G$ be any vector-valued function in $[L^p(\R^d)]^d.$ From~\eqref{resolvent} one easily has
	\begin{equation}\label{starting_estimate}
			\norm{(-\Delta^\ast-z)^{-1} G}_{L^{p'}}
			\leq \frac{1}{\mu} \norm{(-\Delta - \tfrac{z}{\mu})^{-1}G_S}_{L^{p'}} + \frac{1}{\lambda + 2\mu} \norm{(-\Delta - \tfrac{z}{\lambda + 2\mu})^{-1}G_P}_{L^{p'}},
	\end{equation}
	where $G=G_S+ G_P$ is the Helmholtz decomposition of $G.$
	
	We shall explicitly estimate the term involving the $S$-component. The analogous term for the $P$- component can be treated similarly.
	It follows from~\eqref{res_1} that, for any $j=1, 2, \dots, d,$
	\begin{equation*}
		\norm{(-\Delta - \tfrac{z}{\mu})^{-1}G_S^{(j)}}_{L^{p'}}\leq \frac{C_{p,d}}{\mu^{-\frac{d+2}{2} + \frac{d}{p}}} \abs{z}^{{-\frac{d+2}{2} + \frac{d}{p}}} \norm{G_S^{(j)}}_{L^p}. 
	\end{equation*}
	This, along with the sub-additivity of the concave function $\abs{x}^p,$ for $0<p\leq 1,$ and the H\"older inequality for discrete measures, gives
			\begin{equation*}
				\begin{split}
				\norm{(-\Delta -\tfrac{z}{\mu})^{-1} G_S}_{L^{p'}}&=
				\Big(\sum_{j=1}^d \norm{(-\Delta - \tfrac{z}{\mu})^{-1} G_S^{(j)}}_{L^{p'}}^{p'} \Big)^\frac{1}{p'}\leq \frac{C_{p,d}}{\mu^{-\frac{d+2}{2} + \frac{d}{p}}} \abs{z}^{{-\frac{d+2}{2} + \frac{d}{p}}}  \sum_{j=1}^d \norm{G_S^{(j)}}_{L^p} \\ 
				& \leq \frac{C_{p,d}}{\mu^{-\frac{d+2}{2} + \frac{d}{p}}} \abs{z}^{{-\frac{d+2}{2} + \frac{d}{p}}}   d^\frac{1}{p'} \Big(\sum_{j=1}^d \norm{G_S^{(j)}}_{L^p}^{p} \Big)^\frac{1}{p}\\
				&=\frac{C_{p,d}}{\mu^{-\frac{d+2}{2} + \frac{d}{p}}} \abs{z}^{{-\frac{d+2}{2} + \frac{d}{p}}} d^\frac{1}{p'}  \norm{G_S}_{L^p}.
				\end{split}
			\end{equation*}
			The same computation for the term involving the $P$- component provides
			\begin{equation*}
				\norm{(-\Delta - \tfrac{z}{\lambda + 2\mu}) G_P}_{L^{p'}}\leq \frac{C_{p, d}}{(\lambda + 2\mu)^{- \frac{d+2}{2} + \frac{d}{p}}} \abs{z}^{- \frac{d+2}{2} + \frac{d}{p}} d^\frac{1}{p'} \norm{G_P}_{L^p}.
			\end{equation*}
			Plugging the previous two bounds together in~\eqref{starting_estimate}, one has
			\begin{equation*}
				\norm{(-\Delta^\ast-z)^{-1} G}_{L^{p'}}
				\leq C_{p,d,\lambda, \mu} \abs{z}^{{-\frac{d+2}{2} + \frac{d}{p}}} \big( \norm{G_S}_{L^p} + \norm{G_P}_{L^p} \big),
			\end{equation*}
			where $C_{p,d,\lambda,\mu}:= C_{p,d}\, d^\frac{1}{p'} \max\Big\{\mu^{{\frac{d}{2} - \frac{d}{p}}}, (\lambda + 2\mu)^{\frac{d}{2} - \frac{d}{p}} \Big\}.$ 
			
			Hence, estimate~\eqref{Lamé_res_1} follows immediately from the latter as a consequence of~\eqref{orth}. 
			
			 Bound~\eqref{Lamé_st_scattering} can be proved with a few modifications to the argument above, namely using~\eqref{st_scattering} and~\eqref{w_orth} with $p=2$ (notice that $\langle x \rangle^{2\alpha}\in A_2(\R^d)$) instead of~\eqref{res_1} and~\eqref{orth}, respectively.
			Similarly bound~\eqref{Lamé_res_3} follows as a consequence of~\eqref{res_3} and~\eqref{w_orth} with $p=2$.
				
				Now, we turn to the proof of~\eqref{Lamé_res_2}.
				
				Let $W\in A_1(\R^d)\cap \mathcal{L}^{\alpha,p}(\R^d)$ be as in Lemma~\ref{good_property_MC}. Observe that, since $W\in \mathcal{L}^{\alpha,p}(\R^d),$ then~\eqref{res_2} is available. Moreover, as $W\in A_1(\R^d),$ in particular $W\in A_2(\R^d)$ and~\eqref{w_orth} with $p=2$ is valid too.  Therefore, it comes as a slight modification of the argument above proving that the following analogue of~\eqref{Lamé_res_3} for $W$ holds true, namely one has
				\begin{equation}\label{for_W}
					\norm{(-\Delta^\ast -z)^{-1} G}_{L^2(W)}\leq C_{\alpha, p, d, \lambda, \mu} \norm{W}_{\mathcal{L}^{\alpha,p}(\R^d)} \abs{z}^{-1 + \frac{\alpha}{2}} \norm{G}_{L^2(W^{-1})}.
				\end{equation}

				Since $V(x)\leq W(x)$ almost everywhere and as $\norm{W}_{\mathcal{L}^{\alpha,p}(\R^d)}\leq c\norm{V}_{\mathcal{L}^{\alpha,p}(\R^d)},$ using estimate~\eqref{for_W} one gets
				\begin{equation*}
					\begin{split}
					\norm{(-\Delta^\ast -z)^{-1} G}_{L^2(V)}&\leq \norm{(-\Delta^\ast -z)^{-1} G}_{L^2(W)}\\
					&\leq  C_{\alpha, p, d, \lambda, \mu} \norm{W}_{\mathcal{L}^{\alpha,p}(\R^d)} \abs{z}^{-1 + \frac{\alpha}{2}} \norm{G}_{L^2(W^{-1})}\\
					&\leq c\, C_{\alpha, p, d, \lambda, \mu} \norm{V}_{\mathcal{L}^{\alpha,p}(\R^d)} \abs{z}^{-1 + \frac{\alpha}{2}} \norm{G}_{L^2(V^{-1})},
					\end{split}
				\end{equation*}
				which is~\eqref{Lamé_res_2}. This concludes the proof of the theorem.
\end{proof}

\section{Self-adjoint setting}
This section is concerned with the proof of eigenvalue bounds of the form~\eqref{real-valued-V-bound} for the \emph{self-adjoint} perturbed Lamé operator. More precisely, we shall prove the following result.
\begin{theorem}\label{thm:eigenvalue_bound_self}
	Let $V$ be real-valued and and let $\gamma\geq 1/2$ if $d=1,$ $\gamma>0$ if $d=2$ and $\gamma\geq 0$ if $d=3.$ Then any negative eigenvalue $z$ of the perturbed Lamé operator $-\Delta^\ast + V$ satisfies
	\begin{equation}\label{SA_bound_Lamé}
		\abs{z}^\gamma\leq C_{\gamma, d,\lambda, \mu} \norm{V_-}_{L^{\gamma + \frac{d}{2}}(\R^d)}^{\gamma + \frac{d}{2}},
	\end{equation}
	with a constant $C_{\gamma, d, \lambda, \mu}$ independent of $V.$
\end{theorem}
 Here $V_-$ denotes the negative part of $V,$ \emph{i.e.} $V_-(x):=\max\{-V(x),0\}.$

\medskip
Before providing the proof of this theorem, let us comment on the corresponding inequalities of type~\eqref{SA_bound_Lamé} for self-adjoint Schr\"odinger operators. As the remarks provided below then naturally carry over to Lamé operators, the choice of discussing the case of the Laplacian only, finds its reasons solely in the intent of lightening the discussion.  

In the case of $-\Delta + V$ with real-valued potential $V,$ estimate~\eqref{SA_bound_Lamé} was first found by Keller~\cite{Keller} in $d=1$ and, later, generalized to an inequality for the negative eigenvalues power sum known as Lieb-Thirring inequality:
\begin{equation}\label{classical_L-T}
	\sum_{z\in \sigma_d(-\Delta + V)} \abs{z}^\gamma \leq L_{\gamma, d} \norm{V_-}_{L^{\gamma + \frac{d}{2}}(\R^d)}^{\gamma + \frac{d}{2}}, 
\end{equation}        
where $\gamma \geq 1/2$ if $d=1,$ $\gamma>0$ if $d=2$ and $\gamma\geq 0$ if $d\geq 3$ (same conditions as in Theorem~\ref{thm:eigenvalue_bound_self}) (see~\cite{L_T} and~\cite{Weidl,Ro1,Ro2,Cw,L1} for the endpoint cases). 
In passing, observe that bounds on single eigenvalues, like~\eqref{SA_bound_Lamé}, represent a much weaker version of the Lieb-Thirring type inequalities~\eqref{classical_L-T}.

Now, some comment on inequalities~\eqref{SA_bound_Lamé} for $-\Delta + V$ (in fact on the stronger bound~\eqref{classical_L-T}) are listed below (we refer to~\cite{Li_Se}, Chapter 4, for further details).
 
\begin{remark}\label{neg_part}
Contrarily to the case of complex-valued potentials, here, as a consequence of the variational principles (no more available in the non self-adjoint context), only the negative part of $V,$ namely $V_-,$ plays a role. Of course, since $-\Delta$ is a non-negative operator, if $V$ is also non-negative then so is $-\Delta + V$ and therefore the variational characterization of the spectrum guarantees that no negative eigenvalues can occur. If $V$ changes its sign, that is if both the positive and negative part of $V=V_+ - V_-$ are non-trivial, it is true that both parts influence the negative eigenvalues, but as $-\Delta + V\geq -\Delta -V_{-},$ it is a consequence of the minimax principle that an upper estimate for the absolute value of negative eigenvalues of $-\Delta - V_-$ provides automatically the same upper estimate for the negative eigenvalues of the complete hamiltonian $-\Delta + V$ (actually, the same reasoning applies to the eigenvalue power sum). Indeed the effect of $V_+$ on negative eigenvalues is only to increase their size. 
\end{remark}

\begin{remark}
It is not difficult to see that if $z$ is an eigenvalue of $-\Delta + V$ with eigenfunction $\psi,$ then $\phi_\alpha(\cdot):= \psi(\alpha \, \cdot)$ is an eigenfunction of $-\Delta + V_\alpha(x)$ where $V_{\alpha}(\cdot)=\alpha^2 V(\alpha\, \cdot)$ with eigenvalue $\alpha^2 z.$ By a simple scaling argument, this gives that $p=\gamma + d/2$ is the only possible exponent for which an inequality of the following type
\begin{equation*}
	\sum_{z\in \sigma_d(-\Delta + V)}\abs{z}^\gamma \leq L_{\gamma, d} \norm{V_-}_{L^p(\R^d)}^p
\end{equation*}
can hold.
\end{remark}

\begin{remark}
	Let us underline that there are ``natural'' constraints on the validity of inequalities of type~\eqref{classical_L-T} that can be easily justified. We emphasize here the pathological behavior of dimensions $d=1,2.$  It is well known that, due to the lack of a Hardy-type inequality, the free hamiltonian $-\Delta$ is critical in low dimensions, which means that the addition of any arbitrarily small non-trivial negative potential $V$ makes the corresponding perturbed operator $-\Delta + V$ negative, thus ensuring existence of negative eigenvalues. On the other hand, if an inequality of the form~\eqref{classical_L-T} with $\gamma=0$ holds, then the left-hand side would turn into the counting function of negative eigenvalues and so, as a consequence of the aforementioned criticality, it is an integer greater or equal to one for any such potential. On the contrary, since the right-hand side can be made arbitrarily small, for instance, assuming $L_{0,d} \norm{V_-}_{L^{d/2}(\R^d)}^{d/2}<1,$ would give an evident contradiction.  
\end{remark}

\medskip
We can now turn to proof of Theorem~\ref{thm:eigenvalue_bound_self}. As we will show later, it will come as a consequence of the following lemma which provides dimension-dependent estimates for the expectation value $\langle \psi, V_- \psi\rangle:=\int_{\R^d} V_-\abs{\psi}^2\,dx$ of the potential energy $V_-$ in the state $\psi\in H^1(\R^d).$

\begin{lemma}\label{lemma:mean}
	Let $d\geq 1$ and let $V\in L^{\gamma + d/2}(\R^d)$ be a real-valued function. The following estimates for $\langle \psi, V_-\psi \rangle$ hold true.
	\begin{enumerate}[i)]
		\item If $d=1$ and $\gamma \geq \frac{1}{2},$ then
			\begin{equation}\label{mean_1d}
		\langle \psi, V_- \psi\rangle \leq \norm{V_-}_{\gamma + \frac{1}{2}}\norm{\psi}_2^{\frac{2(2\gamma -1)}{2\gamma+1}} \norm{\psi}_{\infty}^{\frac{4}{2\gamma+1}}.
			\end{equation}
		\item If $d=2$ and $\gamma > 0,$  then
			\begin{equation}\label{mean_2d}
				\langle \psi, V_- \psi\rangle \leq \norm{V_-}_{\gamma + 1} \norm{\psi}_{\frac{2(\gamma+1)}{\gamma}}^{2}. 
			\end{equation}
			\item If $d\geq 3$ and $\gamma \geq 0,$ then
				\begin{equation}\label{mean_geq_3d}
					\langle \psi, V_- \psi\rangle \leq \norm{V_-}_{\gamma + \frac{d}{2}}\norm{\psi}_2^{\frac{4\gamma}{2\gamma+d}} \norm{\psi}_{\frac{2d}{d-2}}^{\frac{2d}{2\gamma+d}}.
				\end{equation}
				\end{enumerate}
\end{lemma}
\begin{proof}
	Let us start with the proof of~\eqref{mean_geq_3d}.
	
	It is an easy consequence of H\"older inequality that
	\begin{equation*}
		\langle \psi, V_- \psi \rangle:= \int_{\R^d} V_- \abs{\psi}^2\, dx\leq \norm{V_-}_{\gamma + \frac{d}{2}} \norm{\psi}_{\frac{2(2\gamma + d)}{2\gamma + d-2}}^2.
	\end{equation*}
	
	Being $ 2\leq 2(2\gamma + d)/(2\gamma + d-2)\leq 2d/(d-2),$
				we can use the interpolation inequality to get
				\begin{equation*}
					\norm{\psi}_{\frac{2(2\gamma + d)}{2\gamma + d-2}}^2\leq \norm{\psi}_2^{\frac{4\gamma}{2\gamma + d}} \norm{\psi}_{\frac{2d}{d-2}}^\frac{2d}{2\gamma +d}.
				\end{equation*}
				Plugging the latter in the former gives~\eqref{mean_geq_3d}.
				
				Now let us consider $d=1,2.$ Estimate~\eqref{mean_2d} is immediate consequence of H\"older inequality and the same holds for the case $\gamma=1/2$ in $d=1.$ 
				Finally, the remaining case $\gamma>1/2$ follows, as in the three dimensional framework, from H\"older and interpolation inequality, using that $2\leq 2(2\gamma + 1)/(2\gamma-1)<\infty.$ This proves~\eqref{mean_1d} and concludes the proof of the lemma.
\end{proof}

In passing, observe that if $\psi\in H^1(\R^d),$ then the norms on the right hand side of~\eqref{mean_1d},~\eqref{mean_2d} and~\eqref{mean_geq_3d} are finite. This is a consequence of the Sobolev embeddings
\begin{equation}
\label{Sobolev_embeddings}
H^1(\R^d)\hookrightarrow L^q(\R^d)
\quad \text{where}\quad
	\begin{system}
		& q=\infty \quad &\text{if}\quad d=1,\\
		& 2\leq q<\infty \quad &\text{if}\quad d=2,\\
		& q=2d/(d-2) \quad &\text{if}\quad d\geq 3,
	\end{system}
\end{equation}
which are quantified by the inequalities contained in the following lemma (see~\cite{Li_Lo}, Chapter 8).
\begin{lemma}[Sobolev inequalities]\label{Sobolev_ineq}
Let $d\geq 1$ and let $\psi\in H^1(\R^d).$
	\begin{enumerate}[i)]
		\item If $d=1,$ then
			\begin{equation}\label{Sob1}
				\int_{\R} \abs*{\frac{d\psi}{dx}}^2\, dx \geq \norm{\psi}_2^{-2} \norm{\psi}_\infty^4.
			\end{equation}
		\item If $d=2,$ then
			\begin{equation}\label{Sob2}
				\int_{\R^2} \abs{\nabla \psi}^2\, dx\geq S_{2,q} \norm{\psi}_2^{-\frac{4}{q-2}} \norm{\psi}_q^\frac{2q}{q-2}, \qquad 2< q<\infty.
			\end{equation}
		\item If $d\geq 3,$ then
			\begin{equation}\label{Sob3}
				\int_{\R^d} \abs{\nabla \psi}^2\, dx\geq S_d \norm{\psi}_{\frac{2d}{d-2}}^2.
			\end{equation}
	\end{enumerate}
\end{lemma}

Now we are in position to prove Theorem~\ref{thm:eigenvalue_bound_self}.

\subsection{Proof of Theorem~\ref{thm:eigenvalue_bound_self}}
	In this setting, the variational characterization of the spectrum states that for any $u\in [H^1(\R^d)]^d,$
	\begin{equation*}
		\inf \sigma(-\Delta^\ast + V_-)= \inf_{\norm{u}_{[L^2(\R^d)]^d}=1} \langle u, (-\Delta^\ast -V_-)u \rangle.
	\end{equation*}
	Therefore, in order to get~\eqref{SA_bound_Lamé} it is sufficient to prove the following lower bound
	\begin{equation}\label{sufficient_bound}
		\langle u, (-\Delta^\ast - V_-) u \rangle \geq  - C_{\gamma, d,\lambda, \mu}^\frac{1}{\gamma} \Big( \int_{\R^d} V_-^{\gamma+ \frac{d}{2}}\, dx \Big)^{\frac{1}{\gamma}}
	\end{equation}
	for any $u\in [H^1(\R^d)]^d$ with $\norm{u}_{[L^2(\R^d)]^d}=1.$
	In order to estimate $\langle u, (-\Delta^\ast - V_-) u \rangle,$ we consider the Helmholtz decomposition $u=u_S + u_P$ of $u.$ It follows from the explicit expression~\eqref{simple_writing} of $-\Delta^\ast$ and from the $H^1$-orthogonality of $u_S$ and $u_P$ (see Subsection~\ref{Ssec:Helmholtz}) that 
		\begin{equation}\label{eq:preliminary_SA}
			\begin{split}
			\langle u, (-\Delta^\ast - V_-) u \rangle&=\langle u_S, -\mu\Delta u_S\rangle + \langle u_P, -(\lambda + 2\mu)\Delta u_P\rangle - \langle u, V_-u\rangle\\
			&=\mu \sum_{j=1}^d \int_{\R^d} \abs{\nabla u_S^{(j)}}^2\, dx + (\lambda + 2\mu) \sum_{j=1}^d \int_{\R^d} \abs{\nabla u_P^{(j)}}^2\, dx - \sum_{j=1}^d \langle u^{(j)}, V_-u^{(j)}\rangle.
			\end{split}
		\end{equation}
		Let us start considering the case $d\geq 3.$ 
	
	Using Sobolev inequality~\eqref{Sob3} on the $j$-th component of the vector-field $u_S$, one gets
	\begin{equation}\label{eq:bound_u_S^j}
		\begin{split}
		\sum_{j=1}^d \int_{\R^d} \abs{\nabla u_S^{(j)}}^2\, dx 
		&\geq S_d \sum_{j=1}^d \norm{u_S^{(j)}}_{\frac{2d}{d-2}}^2
		\\
		&\geq S_d \Big(\sum_{j=1}^d \norm{u_S^{(j)}}_{\frac{2d}{d-2}}^\frac{2d}{d-2}\Big)^\frac{d-2}{d}=: S_d \norm{u_S}_{\frac{2d}{d-2}}^2,
		\end{split}
	\end{equation}
	where in the last inequality we used the sub-additivity of the concave function $\abs{x}^\frac{d-2}{d}.$
	
	The same computation performed for $u_P$ gives
	\begin{equation}\label{eq:bound_u_P^j}
		\sum_{j=1}^d \int_{\R^d} \abs{\nabla u_P^{(j)}}^2\, dx\geq S_d \norm{u_P}_{\frac{2d}{d-2}}^2.
	\end{equation}
	
	Now we are in position to estimate $\langle u, V_- u \rangle.$ Using bound~\eqref{mean_geq_3d} in Lemma~\ref{lemma:mean}, two times the H\"older inequality for discrete measures and, finally, the Young inequality, $ab\leq a^p/p + b^q/q$ which holds for all positive $a,b$ and $1/p+1/q=1,$ we get for some $\varepsilon>0$
	\begin{equation}\label{eq:bound-V_-}
		\begin{split}
			\sum_{j=1}^d \langle u^{(j)}, V_- u^{(j)} \rangle
			&\leq \norm{V_-}_{\gamma + \frac{d}{2}} \sum_{j=1}^d \norm{u^{(j)}}_2^\frac{4\gamma}{2\gamma + d} \norm{u^{(j)}}_{\frac{2d}{d-2}}^\frac{2d}{2\gamma + d}
			\leq d^\frac{2}{2\gamma+d} \norm{V_-}_{\gamma + \frac{d}{2}}\norm{u}_2^{\frac{4\gamma}{2\gamma + d}} \norm{u}_{\frac{2d}{d-2}}^\frac{2d}{2\gamma + d}\\
			&\leq \frac{d^\frac{2}{2\gamma+d}}{\varepsilon^{1+ \frac{d}{2\gamma}}} \frac{2\gamma}{2\gamma + d} \norm{V_-}_{\gamma + \frac{d}{2}}^{1+ \frac{d}{2\gamma}} + \varepsilon^{1 + \frac{2\gamma}{d}} \frac{d^{1+\frac{2}{2\gamma+d}}}{2\gamma + d} \norm{u}_{\frac{2d}{d-2}}^2\\
			&\leq \frac{d^\frac{2}{2\gamma+d}}{\varepsilon^{1+ \frac{d}{2\gamma}}} \frac{2\gamma}{2\gamma + d} \norm{V_-}_{\gamma + \frac{d}{2}}^{1+ \frac{d}{2\gamma}} + 2 \varepsilon^{1+ \frac{2\gamma}{d}  } \frac{d^{1+\frac{2}{2\gamma+d}}}{2\gamma + d} \big[ \norm{u_S}_{\frac{2d}{d-2}}^2 + \norm{u_P}_{\frac{2d}{d-2}}^2\big], 
		\end{split}
	\end{equation}
	where in the last inequality we simply used the inequality $\norm{u}_{\frac{2d}{d-2}}^2\leq 2 \big[ \norm{u_S}_{\frac{2d}{d-2}}^2 + \norm{u_P}_{\frac{2d}{d-2}}^2 \big].$
	
	Now, plugging~\eqref{eq:bound_u_S^j},~\eqref{eq:bound_u_P^j} and~\eqref{eq:bound-V_-} in~\eqref{eq:preliminary_SA}, one has
	\begin{multline*}
		\langle u, (-\Delta^\ast -V_-)u \rangle\\ 
		\geq \Big(\min\{\mu, \lambda + 2\mu\} S_d - 2 \varepsilon^{1+ \frac{2\gamma}{d}} \frac{d^{1+ \frac{2}{2\gamma+d}}}{2\gamma + d} \Big) \big( \norm{u_S}_{\frac{2d}{d-2}}^2 + \norm{u_P}_{\frac{2d}{d-2}}^2\big)
		- \frac{d^\frac{2}{2\gamma+d}}{\varepsilon^{1 + \frac{d}{2\gamma}}} \frac{2\gamma}{2\gamma + d} \norm{V_-}_{\gamma + \frac{d}{2}}^{1 + \frac{d}{2\gamma}}. 
	\end{multline*}
	
	Choosing a suitable small $\varepsilon=: \varepsilon_{\gamma, d, \lambda, \mu},$ one gets~\eqref{sufficient_bound} with
	\begin{equation*}
		C_{\gamma, d, \lambda, \mu}^\frac{1}{\gamma}:= \frac{d^\frac{2}{2\gamma+d}}{\varepsilon_{\gamma, d, \lambda, \mu}^{1 + \frac{d}{2\gamma}}} \frac{2\gamma}{2\gamma + d}.
	\end{equation*}
	Hence, bound~\eqref{SA_bound_Lamé} is proved if $d\geq 3.$ We skip the proof of the analogous bounds in the lower dimensional cases, namely $d=1,2.$ Indeed these follow from the corresponding estimates in Lemma~\ref{lemma:mean} and Sobolev inequalities (Lemma~\ref{Sobolev_ineq}) with minor modifications from the reasoning above.

\section{Non self-adjoint setting: Proof of main results}
\label{Sec:Proofs}

This section is concerned with the proof of the eigenvalue bounds contained in Theorem~\ref{1_Lamé}-~\ref{KS_3_Lamé} and~\ref{4_Lamé}. As we will see, with the estimates of Theorem~\ref{thm:unif_resolvent_Lamé} in hand, the proofs will follow smoothly.

\subsection*{Proof of Theorem~\ref{1_Lamé}}
As already mentioned in the introduction, the starting point in our proofs is the Birman-Schwinger principle. In our context it states that if $z\in \C\setminus [0,\infty)$ is an eigenvalue of $-\Delta^\ast + V,$ then $-1$ is an eigenvalue of the Birman-Schwinger operator $V_{1/2} (-\Delta^\ast - z)^{-1} \abs{V}^{1/2}$ on $[L^2(\R^d)]^d.$ This implies that the operator norm of the latter is at least $1.$ Therefore, in order to get the bound~\eqref{bound_eigenvalues-Lamé}, we are reduced to prove  
		\begin{equation}\label{rephrase}
			\norm{V_\frac{1}{2} (-\Delta^\ast -z)^{-1} \abs{V}^\frac{1}{2}}_{L^2\to L^2}^{\gamma + \frac{d}{2}} \leq C_{\gamma, d, \lambda, \mu} \abs{z}^{-\gamma} \norm{V}_{L^{\gamma + \frac{d}{2}}}^{\gamma+ \frac{d}{2}}.
		\end{equation}
		The same strategy, with the needed modifications, then will be also applied to prove the corresponding bounds in Theorem~\ref{2_Lamé}, Theorem~\ref{KS_3_Lamé} and~\ref{4_Lamé}.
		
		Providing bounds for $\norm{V_\frac{1}{2} (-\Delta^\ast -z)^{-1} \abs{V}^\frac{1}{2}}_{L^2\to L^2}$ is estimating the quantity $\abs{\langle f, V_{1/2} (-\Delta^\ast - z)^{-1}\abs{V}^\frac{1}{2} g \rangle},$ for $f, g\in [L^2(\R^d)]^d.$ 
		To simplify the notation we introduce the function $G:= \abs{V}^\frac{1}{2} g.$ Using the H\"older inequality and estimate~\eqref{Lamé_res_1} for the resolvent $(-\Delta^\ast -z)^{-1},$ we have
		\begin{equation*}
			\begin{split}
			\abs{\langle f, V_\frac{1}{2} (-\Delta^\ast - z)^{-1} G\rangle}&\leq \norm{f \abs{V}^\frac{1}{2}}_p \norm{(-\Delta^\ast -z)^{-1} G}_{p'}\\
			&\leq C_{p,d,\lambda, \mu} \abs{z}^{- \frac{d+2}{2} + \frac{d}{p}} \norm{f \abs{V}^\frac{1}{2}}_p \norm{G}_p.
			\end{split}
		\end{equation*}
		Thus, recalling that $G= \abs{V}^{1/2} g,$ one has
		\begin{equation}\label{1_preliminary}
			\abs{\langle f, V_\frac{1}{2} (-\Delta^\ast - z)^{-1} \abs{V}^\frac{1}{2} g \rangle}
			\leq C_{p,d,\lambda, \mu} \abs{z}^{- \frac{d+2}{2} + \frac{d}{p}} \norm{f \abs{V}^\frac{1}{2}}_p \norm{\abs{V}^\frac{1}{2} g}_p.
		\end{equation}
		Using the H\"older inequality and its version for discrete measures and the sub-additivity property of the concave function $\abs{x}^p,$ for $0<p\leq 1,$  one gets
		\begin{equation*}
			\norm{f\, \abs{V}^\frac{1}{2}}_p= \Big(\sum_{j=1}^d \norm{f_j \abs{V}^\frac{1}{2}}_p^p \Big)^\frac{1}{p}\leq \Big( \sum_{j=1}^d \norm{f_j}_2^p \norm{V}_{\frac{p}{2-p}}^\frac{p}{2} \Big)^\frac{1}{p}\leq \norm{V}_{\frac{p}{2-p}}^\frac{1}{2} \sum_{j=1}^d \norm{f_j}_2 \leq d^\frac{1}{2} \norm{V}_{\frac{p}{2-p}}^\frac{1}{2}\norm{f}_2. 
		\end{equation*}
		The same estimate for the term involving $g$ gives
		\begin{equation*}
			\norm{\abs{V}^\frac{1}{2} g}_p\leq d^\frac{1}{2} \norm{V}_{\frac{p}{2-p}}^\frac{1}{2}\norm{g}_2. 
		\end{equation*}
		Plugging these two bounds in~\eqref{1_preliminary} we end up with the following inequality
		\begin{equation*}
			\abs{\langle f, V_\frac{1}{2} (-\Delta^\ast - z)^{-1} \abs{V}^\frac{1}{2} g \rangle}
			\leq C_{p,d,\lambda, \mu} \, d\, \abs{z}^{- \frac{d+2}{2} + \frac{d}{p}}\norm{V}_{\frac{p}{2-p}}\norm{f}_2\norm{g}_2.
		\end{equation*}
		Now, choosing $p=\frac{2(2\gamma + d)}{2\gamma + d + 2}$ (observe that the restriction on $\gamma$ in Theorem~\ref{1_Lamé} guarantees that $p$ satisfies the hypotheses in Theorem~\ref{thm:unif_resolvent_Lamé}) and taking the supremum over all $f, g\in [L^2(\R^d)]^d$ with norm less than or equal to one, we get~\eqref{rephrase}. This concludes the the proof of Theorem~\ref{1_Lamé}. 
\qedhere

\subsection*{Proof of Theorem~\ref{2_Lamé}}
	As in the proof of the previous result, we are reduced to prove the following bound
	\begin{equation*}
		\norm{V_\frac{1}{2} (-\Delta^\ast -z)^{-1} \abs{V}^\frac{1}{2}}_{L^2\to L^2}^{\gamma + \frac{d}{2}} \leq C_{\gamma, d, \lambda, \mu} \abs{z}^{-\gamma}\norm{V}_{\mathcal{L}^{\alpha,p}(\R^d)}^{\gamma + \frac{d}{2}}.
	\end{equation*}
	We pick a strictly positive function $\phi \in \mathcal{L}^{\alpha, p}$ and we define a strictly positive approximation of our potential, that is $V_\varepsilon(x):= \sup\{ \abs{V(x)}, \varepsilon \phi(x)\}.$ Using Cauchy-Schwarz inequality and estimate~\eqref{Lamé_res_2} for the resolvent $(-\Delta^\ast -z)^{-1},$ we have
	\begin{equation*}
		\begin{split}
		\abs{\langle f, V_\frac{1}{2} (-\Delta^\ast-z)^{-1} \abs{V}^\frac{1}{2} g \rangle}
		&\leq \norm{f \sqrt{\abs{V}/ V_\varepsilon}}_2 \norm{(-\Delta^\ast-z)^{-1} \abs{V}^\frac{1}{2} g}_{L^2(V_\varepsilon)}\\
		&\leq C_{\alpha, d, \lambda, \mu} \abs{z}^{-1 + \frac{\alpha}{2}} \norm{V_\varepsilon}_{\mathcal{L}^{\alpha, p}}\norm{f \sqrt{\abs{V}/V_\varepsilon}}_2\, \norm{g\sqrt{\abs{V}/V_\varepsilon}}_2\\
		&\leq C_{\alpha, d, \lambda, \mu} \abs{z}^{-1+ \frac{\alpha}{2}} \norm{V_\varepsilon}_{\mathcal{L}^{\alpha,p}} \norm{f}_2\, \norm{g}_2. 
		\end{split}
	\end{equation*}
	Thus, the theorem is proved once $\varepsilon$ goes to zero, taking the supremum over all $f, g \in [L^2(\R^d)]^d$ with norm less than or equal to one and by choosing $\alpha= \frac{2d}{2\gamma +d}.$
\qedhere

\subsection*{Proof of Theorem~\ref{KS_3_Lamé}}
	Again we are reduced to prove the following bound
	\begin{equation*}
		\norm{V_\frac{1}{2} (-\Delta^\ast -z)^{-1} \abs{V}^\frac{1}{2}}_{L^2\to L^2}^{\gamma + \frac{d}{2}} \leq C_{\gamma, d, \lambda, \mu} \abs{z}^{-\gamma} \norm{V^\beta}_{\mathcal{KS}_{\alpha}}^{\frac{1}{\beta}(\gamma +\frac{d}{2})}.
	\end{equation*}
	
	The same strategy used above, with the usage of~\eqref{Lamé_res_3} instead of~\eqref{Lamé_res_2}, gives
	\begin{equation*}
		\begin{split}
		\abs{\langle f, V_\frac{1}{2} (-\Delta^\ast-z)^{-1} \abs{V}^\frac{1}{2} g \rangle}
		&\leq \norm{f \sqrt{\abs{V}/ V_\varepsilon}}_2 \norm{(-\Delta^\ast-z)^{-1} \abs{V}^\frac{1}{2} g}_{L^2(V_\varepsilon)}\\
		&\leq C_{\alpha, p, d, \lambda, \mu} \abs{z}^{-\frac{\alpha -d+1}{2\alpha - d+1}} \norm{V_\varepsilon^\beta}_{\mathcal{KS}_{\alpha}}^\frac{1}{\beta} \norm{f \sqrt{\abs{V}/V_\varepsilon}}_2\, \norm{g\sqrt{\abs{V}/V_\varepsilon}}_2\\
		&\leq C_{\alpha, p, d, \lambda, \mu} \abs{z}^{-\frac{\alpha -d+1}{2\alpha - d+1}} \norm{V_\varepsilon^\beta}_{\mathcal{KS}_{\alpha}}^\frac{1}{\beta} \norm{f}_2\, \norm{g}_2. 
		\end{split}
	\end{equation*}
	Thus, the theorem is proved once $\varepsilon$ goes to zero, taking the supremum over all $f, g \in [L^2(\R^d)]^d$ with norm less than or equal to one and by choosing $\alpha= \frac{d(d-1)}{d-2\gamma}.$
\qedhere

\subsection*{Proof of Theorem~\ref{4_Lamé}}
	As before, our problem is reduced to proving the following bound
	\begin{equation}\label{last}
		\norm{V_\frac{1}{2} (-\Delta^\ast -z)^{-1} \abs{V}^\frac{1}{2}}_{L^2\to L^2}^q \leq C_{\gamma, \alpha, d, \lambda, \mu} \abs{z}^{-\gamma} \norm{V}_{L^q(\langle x \rangle^{2\alpha}\, dx)}^q,
	\end{equation}
	with $q=2\gamma + (d-1)/2.$
	First, observe that from~\eqref{Lamé_res_1}, with the choice $p=2(d+1)/(d+3)$ one has
	\begin{equation}\label{interp_1}
		\norm{(-\Delta^\ast-z)^{-1}}_{L^p\to L^{p'}}\leq C_{d,\lambda, \mu} \abs{z}^{-\frac{1}{d+1}}.
	\end{equation}
	In passing, notice that~\eqref{bound_eigenvalues-Lamé} was obtained with the usage of~\eqref{res_1} with the choice $p=2(2\gamma + d)/(2\gamma + d+2)$ (see last part of Proof of Theorem~\ref{1_Lamé}). Then $p=2(d+1)/(d+3)$ corresponds to the case $\gamma=1/2$ in~\eqref{bound_eigenvalues-Lamé} which gave the aforementioned decay threshold $2d/(d+1)$. 
	
	We know from~\eqref{Lamé_st_scattering} that
	\begin{equation}\label{interp_2}
		\norm{(-\Delta^\ast-z)^{-1}}_{L^2(\langle x \rangle^{2\alpha}) \to L^2(\langle x \rangle^{-2\alpha})}\leq C_{\alpha, d, \lambda, \mu} \abs{z}^{-\frac{1}{2}}, \qquad \alpha>\frac{1}{2}.
	\end{equation}
		Using Riesz-Thorin interpolation between estimate~\eqref{interp_1} and~\eqref{interp_2}, we get
	\begin{equation*}
		\norm{(-\Delta^\ast-z)^{-1}}_{L^{p_\theta}(\langle x \rangle^{\alpha \theta p_\theta}) \to L^{p_\theta'}(\langle x \rangle^{-\alpha \theta p_\theta'})} \leq C_{\alpha, d, \lambda, \mu}\abs{z}^{-\frac{1-\theta}{d+1} - \frac{\theta}{2}}, \qquad \frac{1}{p_\theta}= \frac{1-\theta}{p} + \frac{\theta}{2},
	\end{equation*}
	with $p=2(d+1)/(d+3),$ $\alpha>1/2$ and $1/p_\theta + 1/p_\theta'=1.$
	
	From this fact it is easy to get
	\begin{equation*}
		\begin{split}
			\abs{\langle f, V_\frac{1}{2} (-\Delta^\ast - z)^{-1} \abs{V}^\frac{1}{2} g\rangle}
			&\leq \norm{f \abs{V}^\frac{1}{2} \langle x \rangle^{\alpha \theta}}_{{p_\theta}} \norm{(-\Delta^\ast-z)^{-1} \abs{V}^\frac{1}{2} g}_{L^{p_\theta'}(\langle x \rangle^{-\alpha \theta p_\theta'})}\\
			&\leq C_{\alpha, d, \lambda, \mu} \abs{z}^{-\frac{1-\theta}{d+1} - \frac{\theta}{2}} \norm{f \abs{V}^\frac{1}{2} \langle x \rangle^{\alpha \theta}}_{{p_\theta}} \norm{ \langle x \rangle^{\alpha \theta}\abs{V}^\frac{1}{2} g}_{{p_\theta}}\\
			& \leq C_{\alpha, d, \lambda, \mu}\, d \abs{z}^{-\frac{1-\theta}{d+1} - \frac{\theta}{2}} \norm{f}_2 \norm{g}_2 \norm{\abs{V} \langle x \rangle^{2\alpha \theta}}_{\frac{p_\theta}{2-p_\theta}}.
		\end{split}
	\end{equation*}
	
	Taking the supremum over all $f,g\in [L^2(\R^d)]^d$ with norm less than or equal to one and raising the resulting inequality to the power $p_\theta/(2-p_\theta)$ gives
	\begin{equation*}
		\norm{V_\frac{1}{2} (-\Delta^\ast -z)^{-1} \abs{V}^\frac{1}{2}}_{L^2 \to L^2}^\frac{p_\theta}{2-p_\theta} \leq C_{\alpha, d, \lambda, \mu} \abs{z}^{-\big(\frac{1-\theta}{d+1} + \frac{\theta}{2}\big) \tfrac{p_\theta}{2-p_\theta}} \norm{\abs{V} \langle x \rangle^{2\alpha \theta}}_{\frac{p_\theta}{2-p_\theta}}^\frac{p_\theta}{2-p_\theta}.
	\end{equation*}
	Here we abuse the notation by using the same symbol $C_{\alpha, d, \lambda, \mu}$ for different constants.
	
	Calling 
	\begin{equation*}
		\gamma:=\left(\frac{1-\theta}{d+1}+\frac{\theta}{2}\right) \frac{p_\theta}{2-p_\theta},
	\end{equation*}
	this clearly gives $\frac{p_\theta}{2-p_\theta}= 2\gamma\, \frac{d+1}{2-\theta + d\theta},$ since we also have $\frac{1}{p_\theta}= \frac{1-\theta}{p} + \frac{\theta}{2},$ with  $p= \frac{2(d+1)}{d+3},$ this leads to the constraint $\theta =1-\frac{d+1}{4\gamma + d-1}.$ With these choices one has
\begin{equation*}
	\norm{V_\frac{1}{2} (-\Delta^\ast -z)^{-1} \abs{V}^\frac{1}{2}}_{L^2\to L^2}^q \leq C_{\gamma, \alpha, d, \lambda, \mu} \abs{z}^{-\gamma} \norm{V}_{L^q(\langle x \rangle^{2\alpha(\gamma-1)}\, dx)}^q,
\end{equation*}	
with $q=2\gamma + (d-1)/2,$ which is the bound~\eqref{last} once renaming $\alpha(2\gamma -1)=\alpha.$ 
\qedhere



\end{document}